\documentclass[letterpaper, 10 pt, conference]{ieeeconf}  

\IEEEoverridecommandlockouts                              

\usepackage{mathrsfs}
\usepackage{xcolor}

\usepackage{amsmath,amssymb,amsfonts}
\usepackage{graphicx}
\usepackage{algorithm,algorithmic}

\usepackage{textcomp}

\usepackage{setspace}
\usepackage{lipsum}

\usepackage{mathtools}
\usepackage{enumerate}

\usepackage{caption}
\usepackage{subcaption}
\usepackage{algorithm}
\usepackage{float}
\newtheorem{definition}{Definition}
\newtheorem{theorem}{Theorem}

\newtheorem{proposition}{Proposition}
\usepackage{hyperref}

\makeatletter
\let\NAT@parse\undefined
\makeatother
\usepackage{cite}

\newtheorem{example}{Example}

\newcommand{\sgn}{\mathrm{sgn}}
\title{\LARGE \bf
Inverse Optimal Feedback and Gain Margins for Unicycle Stabilization
}

\author{Kwang Hak Kim, Velimir Todorovski, and Miroslav Krstić
\thanks{This work was supported in part by the Office of Naval Research under Grant No. N00014-23-1-2376, in part by the Air Force Office of Scientific Research under Grant No. FA9550-23-1-0535, and in
part by the National Science Foundation under Grant No. ECCS-2151525. The results and opinions in this paper are solely of the authors and do not reflect the position or the policy of the U.S. Government or the National Science
Foundation.}
\thanks{K. Kim, V. Todorovski, and M. Krstić are with the Department of Mechanical and Aerospace Engineering, UC San Diego, 9500 Gilman Drive, La Jolla, CA, 92093-0411, {\tt\small \{kwk001,vtodorovski,krstic\}@ucsd.edu}}%
}

\begin{document}

\maketitle
\thispagestyle{empty}
\pagestyle{empty}


\begin{abstract}
The recent development of globally strict control Lyapunov functions (CLFs) for the challenging unicycle parking problem provides a foundation for pursuing optimality. We address this in the inverse optimal framework, thereby avoiding the need to solve the Hamilton–Jacobi–Bellman (HJB) equations, and establish a general result that is optimal with respect to a meaningful cost. We present several design examples that impose varying levels of penalty on the control effort, including arbitrarily bounded control. Furthermore, we show that the inverse optimal controller possesses an infinite gain margin thanks to the system being driftless, and leveraging this property, we extend the design to an adaptive controller that handles model uncertainty. Finally, we compare the performance of the non-adaptive inverse optimal controller with its adaptive counterpart.
\end{abstract}


\section{Introduction}

The unicycle model serves as a fundamental abstraction for a wide range of robotic and vehicle systems for multitudes of applications \cite{de2002control,campion1996structural}. Its simplicity makes it a crucial benchmark for control design, yet it also captures the essential nonholonomic constraints present in many real-world applications. Despite the unicycle being globally controllable, asymptotic stabilization to a desired configuration (parking) remains difficult due to Brockett’s necessary condition~\cite{brockett1983asymptotic} formally ruling out the existence of any smooth, time-invariant feedback laws.

To overcome this obstruction, previous works have approached with methods such as time-varying feedback \cite{pomet1992explicit,coron1993smooth}, discontinuous feedback \cite{de2000stabilization,bloch1996stabilization_slidingmode}, and even hybrid strategies \cite{hespanha1999_hybrid_stabilization,prieur2003robust}. Another commonly used strategy is the polar coordinate transformation, which circumvents Brockett’s condition by introducing a singularity into the system and has inspired numerous works to develop stabilizing control laws \cite{badreddin1993fuzzy,aicardi1995,astolfi1999exponential}. The polar transformation not only enables stabilization but also supports the construction of strict CLFs, which certify stability with stronger guarantees and have since been discovered in more recent works \cite{todorovski2025_CLF,wang24_force_controlled_safestable,restrepo2020leader}. The availability of strict CLFs has further-reaching implications: the pursuit of optimal controllers. 

In classical optimal control, not only is solving the Hamilton-Jacobi-Bellman (HJB) equation hardly ever possible, but the inescapable `curse of dimensionality' is to an equal degree about storing the results as it is about solving the HJB PDE. For instance, on devices with less than a terabyte of storage, there is not enough space to store the HJB solution on a uniformly quantized grid for a robotic manipulator with more than 3DOF. This motivation led Kalman \cite{kalman1964inverse_opt} to formulate the inverse optimal control problem for linear systems, where one first constructs a stabilizing controller in the form of a gradient of a CLF and then increases its gain by at least twice to obtain a controller that minimizes a meaningful cost functional.

Originally introduced for nonlinear systems in \cite{moylan1973nonlinear}, the nonlinear inverse optimal approach has since been extended in various directions, including to stochastic systems \cite{deng1997stochastic} and to differential games with bounded disturbances \cite{freeman2008robust}, later generalized to arbitrary disturbances through input-to-state stabilizing controllers \cite{krstic1998inverse}. The first constructive exploration of nonlinear inverse optimal control appears in \cite{sepulchre2012constructive}, for feedback laws of the form $-(L_gV)^{\rm T} = - g^{\rm T}\nabla V $, where $g$ is the input vector field of a system affine in control. This approach has found success in applications such as attitude control of rigid spacecraft \cite{krstic1999inverse,bharadwaj1998geometry}, yet remains considerably underdeveloped for the unicycle parking problem. Notably, the work of \cite{do2015global} addresses the stabilization of stochastic nonholonomic systems within the inverse optimal framework; however, the inverse optimality is established only at the subsystem level and without strict CLFs thanks to the inherent properties of stochastic systems. 

In this paper, we exploit the existence of globally strict CLFs to develop general inverse optimal controllers. The resulting feedback achieves even larger gain margins compared to conventional optimal controllers because of the absence of drift in the system. We present several design choice examples that impose different levels of penalty on the control effort. Finally, we extend the framework to adaptive control, demonstrate its ability to handle model uncertainty, and compare the performance with non-adaptive inverse optimal controllers.

\section{$L_gV$ (gradient) controllers}

Consider the polar coordinate representations of the unicycle model given as
\begin{align}\label{eq:unicycle_polar}
\begin{bmatrix}
\dot\rho\\ \dot\delta \\ \dot\gamma
\end{bmatrix} &=
\underbrace{\begin{bmatrix}
-\cos\gamma \\ \frac{\sin\gamma}{\rho} \\ \frac{\sin\gamma}{\rho}
\end{bmatrix}}_{g_1}v + \underbrace{\begin{bmatrix}
0\\0\\-1
\end{bmatrix}}_{g_2}\omega\,,
\end{align}
where $\rho > 0$ is the distance to the origin, $\delta,\gamma \in \mathbb{R}$ are the polar and line-of-sight angles, respectively, and $v$ and $\omega$ represent the forward velocity and steering control input, respectively. The transformations of the angles $\delta$ and $\gamma$ from Cartesian coordinates are defined as $\delta = \mbox{mod}(\text{{\rm atan2}}(y ,x ),2\pi) - \pi$ and $\gamma = \mbox{mod}(\text{{\rm atan2}}(y ,x )-\theta, 2\pi) - \pi$.

Inverse optimal control relies on controllers of the form $-L_gV^\top$ derived from strict CLFs $V$. In that regard, we provide one such CLF in the following proposition. However, we emphasize that any other strict CLFs may be used.

\begin{proposition}[\!\cite{todorovski2025_CLF},{\cite[Sec.~III.A]{restrepo2020leader}}\,] For the control system \eqref{eq:unicycle_polar} the positive definite, radially unbounded function
\begin{align}\label{eq:globa_CLF}
V(\rho,\delta,\gamma) = \rho^2 + \delta^2 + \left(\gamma + \frac{1}{2}\arctan(2 \delta)  \right)^2 
\end{align}
is a strict CLF on $\{\rho > 0\}\times \mathbb{R}^2$.
\end{proposition}

\begin{proof}
Proofs are omitted in this paper for brevity and will appear in an extended journal version.
\end{proof}

Nevertheless, directly applying the $L_gV$ form to \eqref{eq:unicycle_polar} has a singularity at $\rho=0$ through $g_1(\rho,\gamma)$. To remove this, we redefine the velocity input by replacing $v$ with $v/\rho$, forcing $v$ to vanish as $\rho\to 0$. This adjustment is physically intuitive: velocity should decrease in proportion to distance from the target. With the input pair $(v,\omega)$ replaced by $(v/\rho,\omega)$, the new input vector field becomes $\bar g = [\rho g_1, g_2]$, so that $L_{\bar g}V$ is ``desingularized.''

Another key property that the inverse optimal design relies on is that $L_{\bar g}V(\Xi) = 0$ only at $\Xi=0$, where $\Xi := (\rho,\gamma,\delta)$, implying $|L_{\bar g}V(\Xi)|$ is positive definite in $\Xi$. This follows from a general fact: for affine systems $\dot x = f(x) + g(x)u$, a radially unbounded differentiable $V(x)$ is a CLF if and only if $L_gV(x)=0$ for $x\neq 0$ implies $L_fV(x)<0$. When $f(x)\equiv 0$, this condition reduces to $L_gV(x)\neq 0$ for all $x\neq 0$. For our case, the resulting expressions for $-L_{\bar g_1}V$ and $-L_{g_2}V$ obtained from \eqref{eq:unicycle_polar} and \eqref{eq:globa_CLF} are given as

\begin{subequations}
\label{eq:LgV_GloBa-both}
\begin{align}
-L_{\bar{g}_1}V/2 &= \rho^2\cos\gamma  - \sin\gamma\left[\delta + \left(1+\frac{1}{1+4\delta^2}\right)z\right] \label{eq:Lg1V_GloBa}\\
-L_{g_2}V/2 &= z
\,,
\label{eq:Lg2V_GloBa}
\end{align}
\end{subequations}
where $z = \gamma + \frac{1}{2} \arctan(2k_2 \delta)$. While the strict CLF \eqref{eq:globa_CLF} establishes that $|L_{\bar g} V(\Xi)|$ is positive definite in $\Xi$, it is not radially unbounded, a property obtained through the CLF modification described in the proof of Theorem 3.2 in \cite{krstic1998inverse}.

\section{Inverse Optimal Stabilization}

\subsection{Basic quadratic inverse optimality}\label{sec-basic-quadratic}

Before presenting the general inverse optimality methodology in Section \ref{sec-gen-inv-opt}, we give a particular example of what we are after. For \eqref{eq:globa_CLF}, all controllers of the form 
\begin{subequations}\label{eq-u*-quad0}  
\begin{align}
\label{eq-LQ-forward}
v^* &= -\rho\varepsilon_1^2 L_{\bar{g}_1}V\\
\label{eq-LQ-steer}
\omega^* &= -\varepsilon_2^2L_{g_2}V 
\,,
\end{align}
\end{subequations}
for all $\varepsilon_1, \varepsilon_2>0$, are not only globally asymptotically stabilizing but are the minimizers of the parametrized costs
\begin{align}
J = \int_0^\infty &\Biggl[l(\rho,\delta,\gamma) +\left(\dfrac{v}{\varepsilon_1\rho}\right)^2 +\left(\dfrac{\omega}{\varepsilon_2}\right)^2  \Biggr] dt\,,
\end{align}
where $l(\rho,\delta,\gamma) = \left(\varepsilon_1 L_{\bar{g}_1}V\right)^2 +\left(\varepsilon_2L_{g_2}V\right)^2$. Such costs are meaningful in the sense of imposing (i) a zero penalty when all the state variables $\rho,\delta,\gamma$ and both of the ($\rho$-weighted) control inputs $v/\rho$ and $\omega$ are zero, and (ii) a positive penalty when any of the states and either of the (weighted) controls are nonzero. Note that when $\rho = 0$, the forward velocity input is subjected to an infinite penalty, which is desirable, as it discourages  movement away from the positional origin once the origin is reached.

\subsection{General inverse optimal designs}\label{sec-gen-inv-opt}

To generalize the inverse optimality result from Section \ref{sec-basic-quadratic}, we first recall the \textit{Legendre–Fenchel transform} and present our main result.

\begin{definition}(Legendre-Fenchel transform)\label{def:legendre_Fenchel}. Let $\eta$ be a class $\mathcal{K}_\infty[0,a)$ function whose derivative $\eta'$ is also a class $\mathcal{K}_\infty[0,a)$ function where $a > 0$ is finite or infinite.
The mapping 
\begin{align}
\ell \eta(r) = \int_0^r (\eta')^{-1}(s) ds
\end{align}
represents the Legendre-Fenchel transform, where $(\eta')^{-1}(r)$ stands for the inverse function of $d\eta(r)/dr$.
\end{definition} 

\begin{figure*}[t]
\centering
\begin{subfigure}[t]{0.24\textwidth}
\centering
\includegraphics[width=\textwidth]{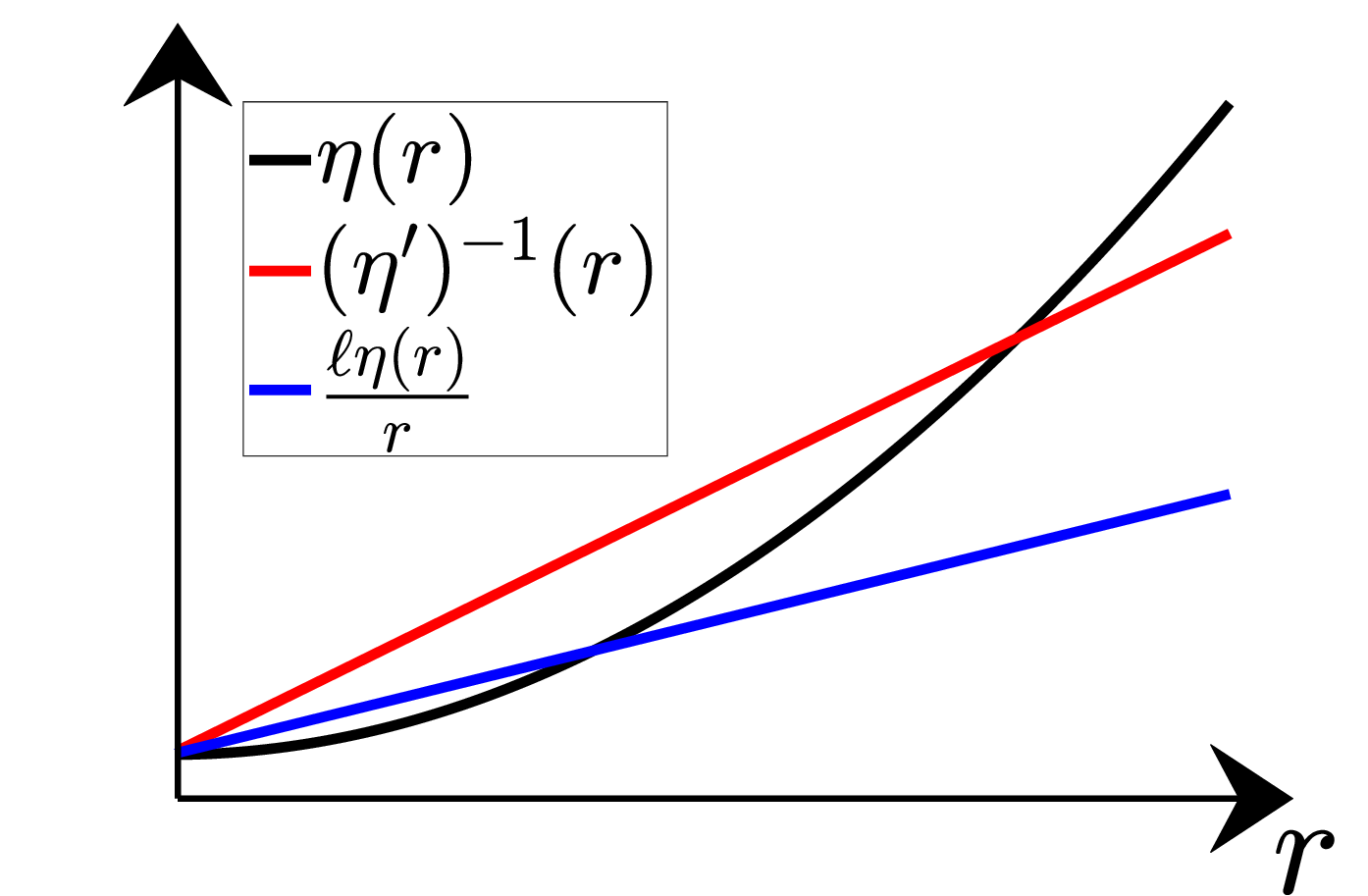}
\caption{$\eta(r) = r^2 / 2$ in (black), $(\eta^{\prime})^{-1}(r) = r$ in (red) and $\frac{\ell \eta (r)}{r} = r/2$ in (blue).}
\label{fig:eta_1}
\end{subfigure}
\begin{subfigure}[t]{0.24\textwidth}
\centering
\includegraphics[width=\textwidth]{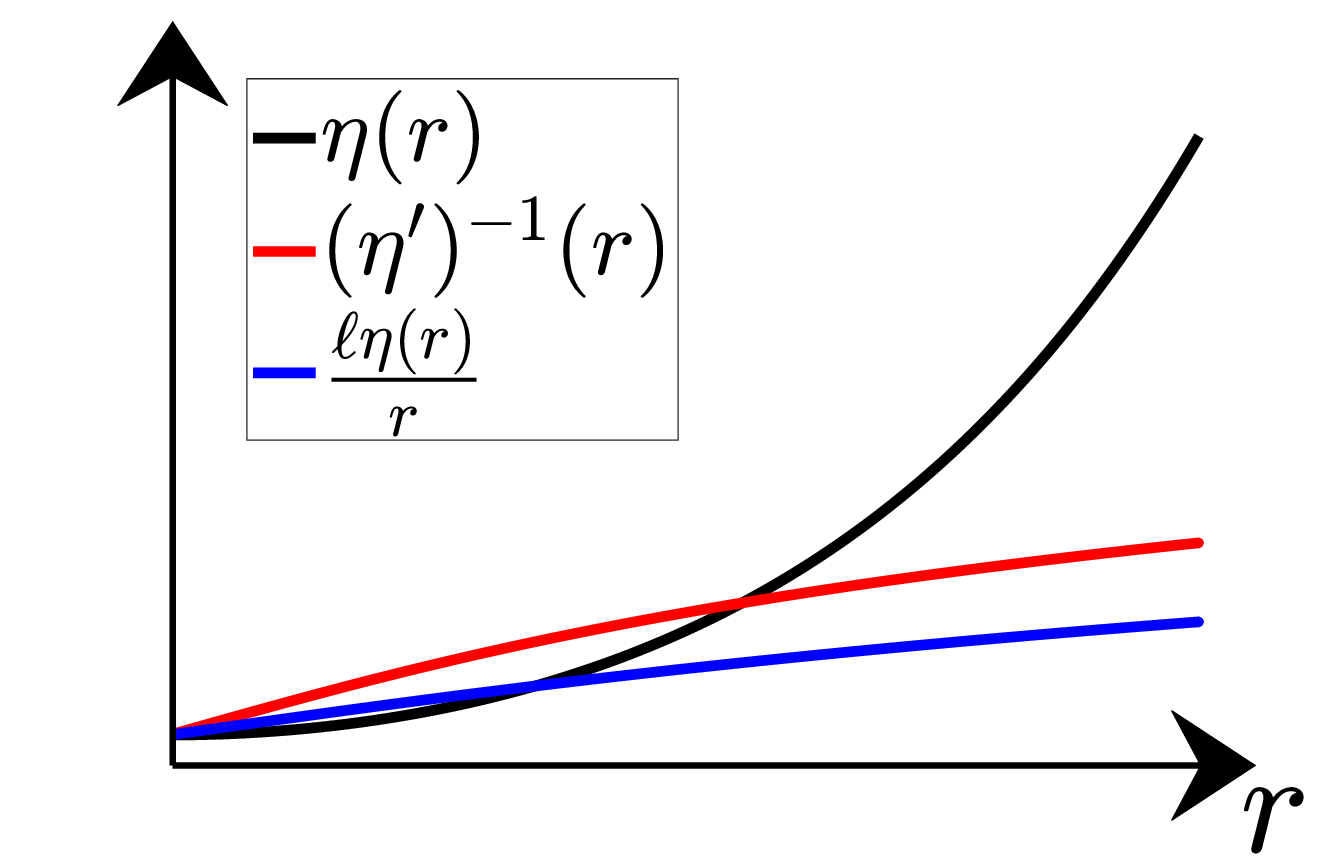}
\caption{$\eta(r) = \cosh(r) - 1$ in (black), $(\eta^{\prime})^{-1}(r) = {\rm arcsinh}(r)$ in (red) and $\frac{\ell \eta (r)}{r} = {\rm arcsinh}(r) - r/(1+\sqrt{r^2+1})$ in (blue).}
\label{fig:eta_2}
\end{subfigure}
\begin{subfigure}[t]{0.24\textwidth}
\centering
\includegraphics[width=\textwidth]{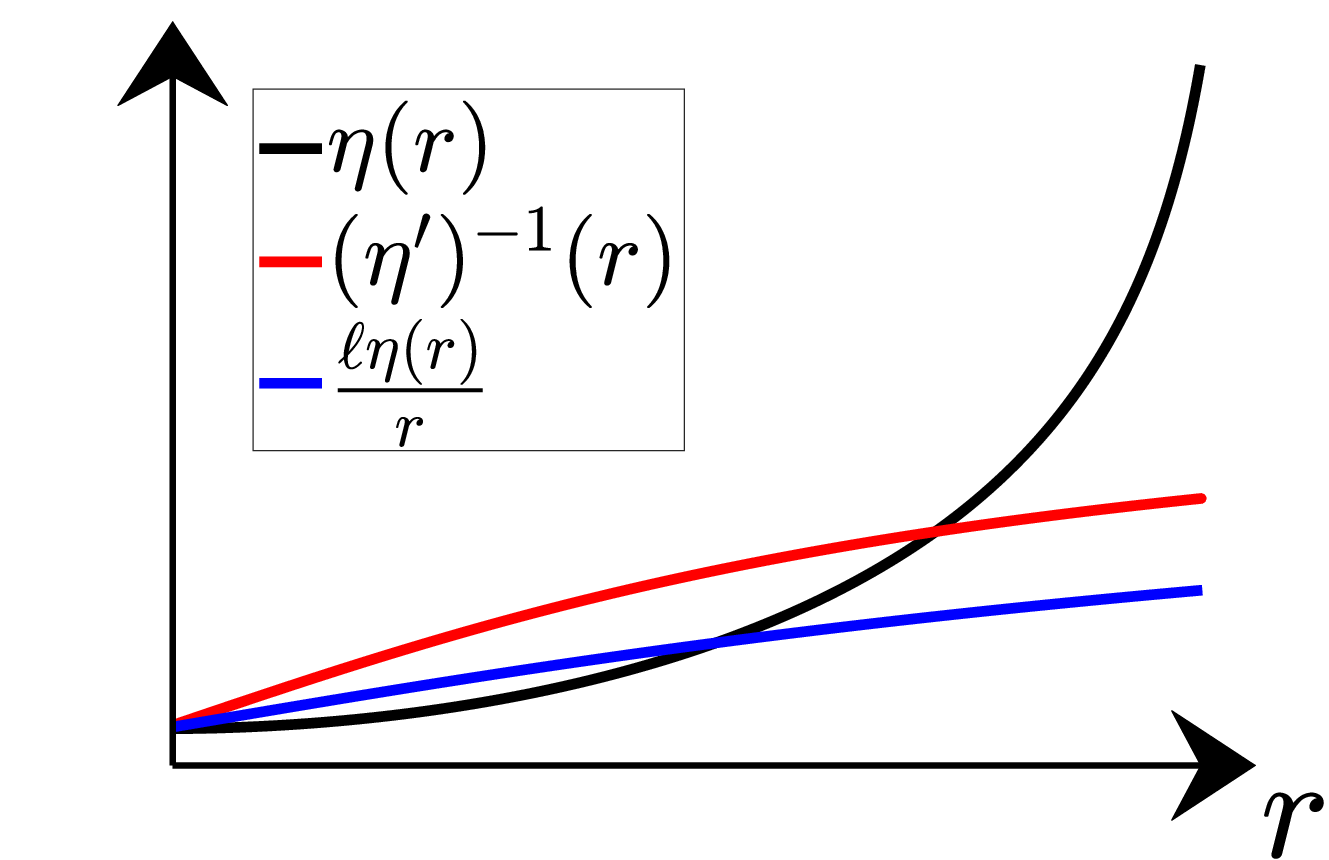}
\caption{$\eta(r) = -\ln(\cos(r))$ in (black), $(\eta^{\prime})^{-1}(r) = \arctan(r)$ in (red) and $\frac{\ell \eta (r)}{r} = \arctan(r) - \ln(1+r^2)/2r$ in (blue).}
\label{fig:eta_3}
\end{subfigure}
\begin{subfigure}[t]{0.24\textwidth}
\centering
\includegraphics[width=\textwidth]{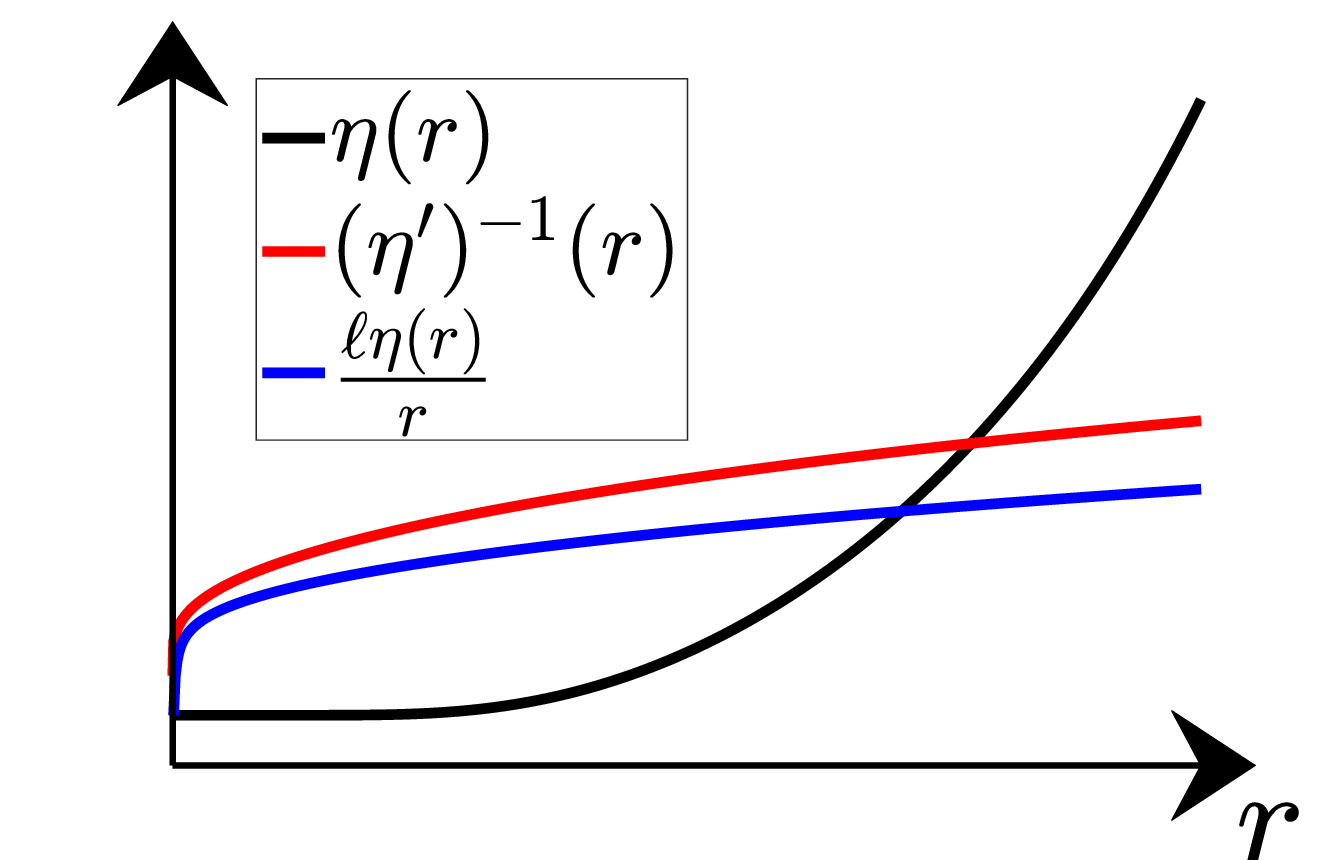}
\caption{$\eta(r) = e/(e^{1/r} - e)$ in (black), $(\eta^{\prime})^{-1}(r) = 1/(1+\ln(1+1/r))$ in (red) and $\frac{\ell \eta (r)}{r} = \frac{1}{r} \int \frac{{\rm d}r}{1+\ln(1+1/r)}$ in (blue).}
\label{fig:eta_4}
\end{subfigure}
\caption{Cost-on-control functions $\eta(r)$, the respective derivative inberse $(\eta')^{-1}(r)$ and the Legendre-Fenchel transform divided by its argument $\ell\eta(r)/r$.}
\label{fig:cost-on-control functions}
\end{figure*}

\begin{theorem} \label{thrm:IOC}Consider the system \eqref{eq:unicycle_polar} rewritten as
\begin{align}\label{eq:ioc_sys}
\dot{\Xi} = \bar{g}_1(\Xi)\frac{v}{\rho} + g_2(\Xi)\omega
\end{align}
where $\Xi \coloneqq [\rho,\delta,\gamma]^\top$ and
\begin{align}
\bar{g}_1 = \begin{bmatrix}
-\rho\cos\gamma\\
\sin\gamma\\
\sin\gamma
\end{bmatrix}, \quad
g_2 = \begin{bmatrix}
0\\
0\\
-1
\end{bmatrix}\,.
\end{align}
For any $\eta_i \in \mathcal{K}_\infty[0,a_i)$ such that also $\eta_i' \in \mathcal{K}_\infty[0,a_i)$ where $a_i > 0$ is finite or infinite,
and for any continuous positive scalar-valued functions $\varepsilon_1(\rho, \delta, \gamma)$ and $\varepsilon_2(\rho, \delta, \gamma)$, the cost functional
\begin{align}\label{eq:ioc_J_thrm}
J = \int_0^\infty &\Biggl[l(\rho,\delta,\gamma) + \eta_1\left(\frac{|v|}{\varepsilon_1\rho}\right) + \eta_2\left(\frac{|\omega|}{\varepsilon_2}\right)\Biggr] dt\,,
\end{align}
where $l(\rho,\delta,\gamma) = \ell\eta_1(\varepsilon_1|\nu_1|) + \ell\eta_2(\varepsilon_2|\nu_2|)$, is minimized by the feedback law
\begin{subequations}
\label{eq:ioc_u*}
\begin{align}
v^* &= -\rho\varepsilon_1 (\eta_1')^{-1}(\varepsilon_1|\nu_1|)\sgn(\nu_1)\\
\omega^* &= -\varepsilon_2 (\eta_2')^{-1}(\varepsilon_2|\nu_2|)\sgn(\nu_2)\,,
\end{align}
\end{subequations}
with $\nu_1(\Xi) \coloneqq L_{\bar{g}1}V$ and $\nu_2(\Xi) \coloneqq L_{g_2}V$ denoted for the CLF \eqref{eq:globa_CLF}, and  
for all initial conditions on $\{\rho>0\}\times\mathbb{R}^2$. Additionally, the feedback law
\begin{subequations}
\label{eq:ioc_u}
\begin{align}
v &= -\rho\varepsilon_1 \frac{\ell\eta_1(\varepsilon_1|\nu_1|)}{\varepsilon_1|\nu_1|}\sgn(\nu_1)\\
\omega &= -\varepsilon_2\frac{\ell\eta_2(\varepsilon_2|\nu_2|)}{\varepsilon_2|\nu_2|}\sgn(\nu_2)\,,
\end{align}
\end{subequations}
is continuous in $(\nu_1,\nu_2)$ and renders the origin $\rho=\delta = \gamma = 0$ of the system \eqref{eq:ioc_sys} globally asymptotically stable on $\{\rho>0\}\times\mathbb{R}^2$.
\end{theorem}

The design space of the stabilizer \eqref{eq:ioc_u} and its optimal counterpart \eqref{eq:ioc_u*} has infinitely many configurations. An exhaustive exploration is impossible, but it is possible to explore qualitatively the tradeoff between the control effort and the related cost on the state.


\section{Specific inverse optimal choices}


\begin{figure}[t]
\centering
\includegraphics[width=\linewidth]{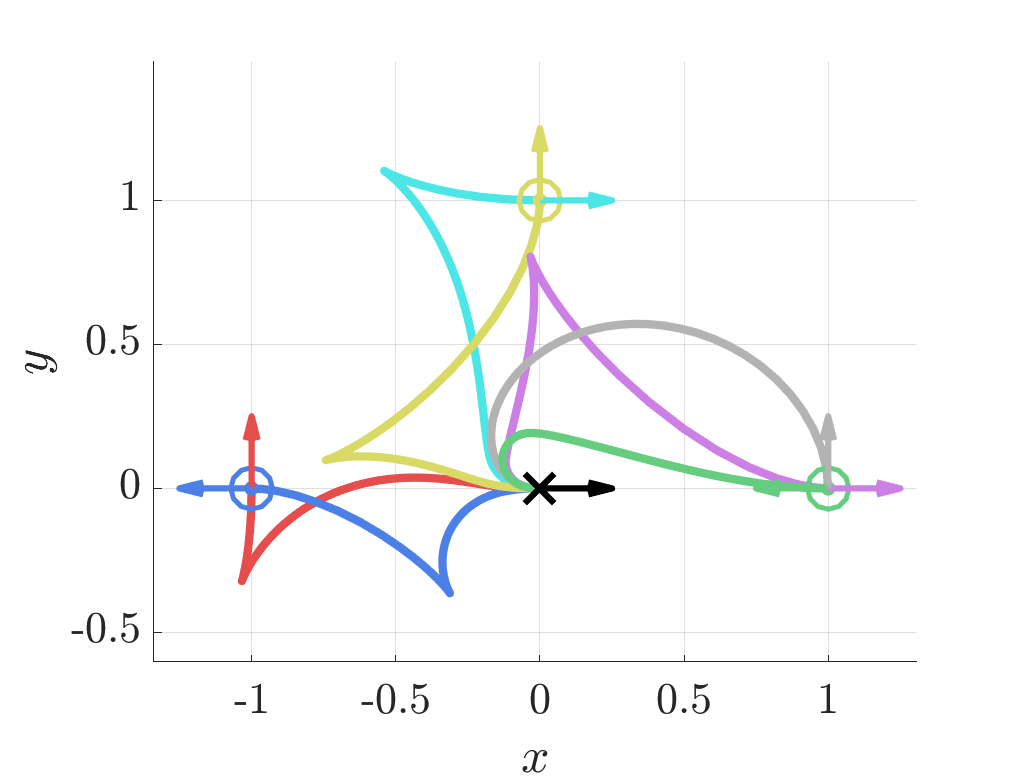}
\caption{Inverse optimal controller trajectory with a quadratic cost on control effort \eqref{eq:eta_quad} for several representative initial conditions (respective colors) to the target position and heading (black).}
\label{fig:ioc_trajectory}
\end{figure}

We present four examples based on different choices of the function $\eta$, each leading to a distinct inverse optimal controller. Example~\ref{example:IOC1} yields a controller linear in $L_{\bar{g}_1}V$ and $L_{g_2}V$, already introduced in Section \ref{sec-basic-quadratic}, Example~\ref{example:IOC2} produces a sublinear relation of control with  $L_{\bar{g}_1}V$ and $L_{g_2}V$, and Examples~\ref{example:IOC3} and \ref{example:IOC4} result in bounded control laws. While Theorem \ref{thrm:IOC} does not require uniformity in the choice of $\eta_1$ and $\eta_2$, the following examples use the same function for both to highlight each behavior, and  the subscript on $\eta(r)$ is omitted henceforth. The four choices of $\eta$, along with their corresponding derivative inverses and Legendre–Fenchel transformations, are summarized in Fig.~\ref{fig:cost-on-control functions}, illustrating how the associated control effort costs increase.


\begin{example}\label{example:IOC1} \textbf{(Linear in $L_{\bar{g}_1}V$ and $L_{g_2}V$)}
Consider the quadratic cost function defined by
\begin{align}\label{eq:eta_quad}
\eta(r) = \frac{r^2}{2}
\,,
\end{align} 
corresponding to the classical quadratic cost on both the state running cost and the control, as detailed in the basic quadratic inverse optimal controller \eqref{eq-u*-quad0}.

Simulation results for $\varepsilon_1 = \varepsilon_2 = 1$ (Fig.~\ref{fig:ioc_trajectory}) show that, for many initial conditions, the controller first reverses to realign before moving toward the target. For the case where the unicycle starts parallel and directly above the target (cyan), the motion resembles a reversed parallel parking maneuver, similar to how a human driver might approach the task.

However, quadratic costs on control and state may be naive for nonlinear systems in general, particularly for the unicycle. Although they penalize input over the time horizon, they may still yield occasional large values outside actuator limits. As shown in Fig.~\ref{fig:ioc_controleffort_ex1}, this results in forward velocity magnitudes up to $7$ units/s and steering rate magnitudes up to $8$ units/s---values not unreasonable but potentially beyond actuator range. In the next three examples, we present alternatives with stronger control penalties. Since the trajectories $y(x)$ remain nearly identical across cases, we focus instead on the differences in control effort.
\end{example}

\begin{example}\label{example:IOC2} \textbf{(Sublinear/logarithmic in $L_{\bar{g}_1}V$ and $L_{g_2}V$)}
Consider the hyperbolic cosine cost function defined as
\begin{align}\label{eq:eta_hyperbolic}
\eta(r) = \cosh(r) - 1
\,,
\end{align}
which imposes a larger penalty when the control effort is large as compared to the quadratic cost in Example~\ref{example:IOC1}. This is reflected in the minimizer feedback
\begin{subequations}\label{eq-u*-cosh}
\begin{align}
v^* &= -\rho\varepsilon_1\mathrm{arcsinh}\left(\varepsilon_1L_{\bar{g}_1}V\right)\\
\omega^* &= -\varepsilon_2\mathrm{arcsinh}\left(\varepsilon_2L_{g_2}V\right)\,,
\end{align}
\end{subequations}
where, compared to the linear optimal controller in \eqref{eq-u*-quad0}, the feedback law in \eqref{eq-u*-cosh} exhibits a sublinear dependence on $L_{\bar{g}_1}V$ and $L_{g_2}V$. This implies that the controller grows more slowly than linearly with respect to $L_{\bar{g}_1}V$ and $L_{g_2}V$. Thus, the feedback in \eqref{eq-u*-cosh} is interpreted as a less effort-intensive alternative to the linear controller in \eqref{eq-u*-quad0}. This is reflected in Fig~\ref{fig:ioc_controleffort_ex2}, with $\varepsilon_1 = \varepsilon_2 = 1$, showing a nearly threefold reduction in control effort.
\end{example}

\begin{figure}[t]
\centering
\begin{subfigure}{0.48\textwidth}
\centering
\includegraphics[width=\textwidth]{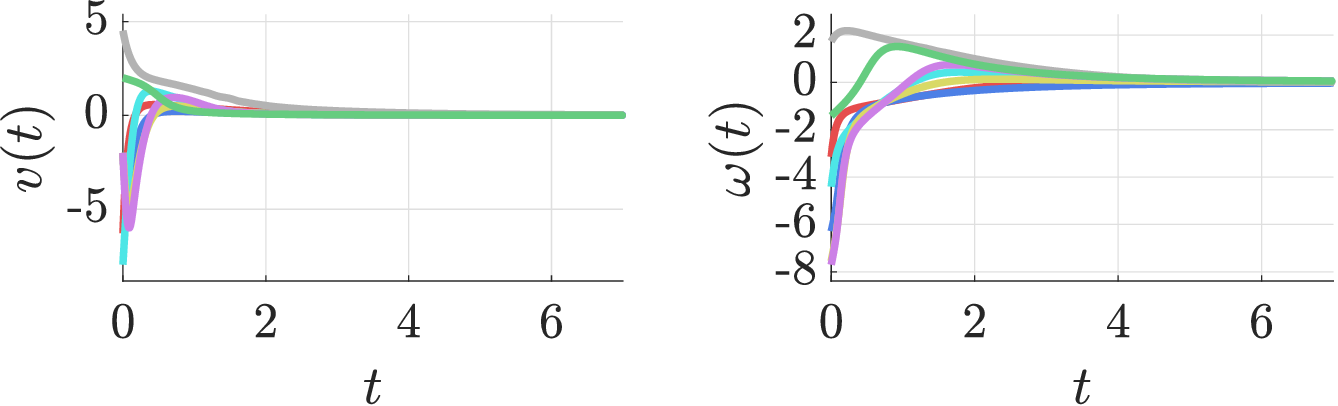}
\caption{Control effort of the linearly proportional to $L_{\bar{g}_1}V$ and $L_{g_2}V$ controller \eqref{eq-u*-quad0}.}
\label{fig:ioc_controleffort_ex1}
\end{subfigure}
\vspace{0.25cm}
\begin{subfigure}{0.48\textwidth}
\centering
\includegraphics[width=\textwidth]{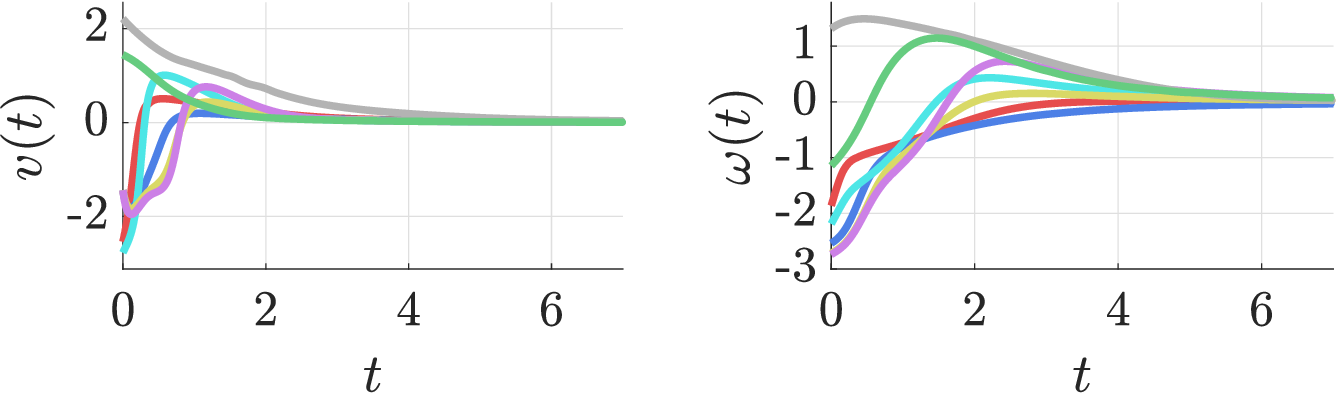}
\caption{Control effort of the sublinearly proportional to $L_{\bar{g}_1}V$ and $L_{g_2}V$ controller \eqref{eq-u*-cosh}.}
\label{fig:ioc_controleffort_ex2}
\end{subfigure}
\caption{Comparison in control effort between Example~\ref{example:IOC1} and Example~\ref{example:IOC2} with $\varepsilon_1 = \varepsilon_2 = 1$.}
\label{fig:ioc_ex1_vs_ex2}
\end{figure}

\begin{example}\label{example:IOC3}\textbf{(Bounded/saturating control)} While the hyperbolic cost on control effort significantly reduced the input magnitudes, one might still desire a stronger guarantee by directly enforcing a bound on the maximum allowable control input. 
To achieve bounded optimal control, we require that the inverse of $\eta'$ in the general expression for the optimal control be a bounded function. For instance, setting $(\eta')^{-1}(r) = \arctan(r)$ leads to
\begin{equation}\label{eq:eta_ln}
\eta(r) = - \ln(\cos(r))
\,,
\end{equation}
which is of class $\mathcal{K}_\infty$ on the interval $[0,\pi/2)$. 

This choice implies that the control cost becomes unbounded as $|v|$ approaches the limit $\rho\varepsilon_1\pi/2$ and $|\omega|$ approaches the limit $\varepsilon_2 \pi/2$, as seen in the optimal controller given by
\begin{subequations}\label{eq:u*_ln}
\begin{align}
v^* &= -\rho\varepsilon_1\arctan\left(\varepsilon_1L_{\bar{g}_1}V\right)\\
\omega^* &= -\varepsilon_2\arctan\left(\varepsilon_2L_{g_2}V\right)\,.
\end{align}
\end{subequations}
While the upper bound $|\omega| \leq \overline{\omega}$ is easily enforced by selecting $\varepsilon_2 = 2\overline{\omega}/\pi$, the upper bound $|v| \leq \overline{v}$ is more subtle. This is because the limiting value $\rho \varepsilon_1 \pi/2$ is dependent on $\rho$. To address this, we define
\begin{align}
\label{eq-eps1-rho}
\varepsilon_1(\rho) = \frac{\overline{v}}{\sigma + \rho} \frac{2}{\pi} > 0\,,
\end{align}
where $\sigma > 0$ is a small constant, to ensure the bound holds. Fig~\ref{fig:ioc_controleffort_ex3} illustrates the result for $\bar{v} = \bar{\omega} = 1$ and $\sigma = 0.1$, showing that neither control input exceeds magnitude $1$.
\end{example}

\begin{example}\label{example:IOC4}\textbf{(``Relay-approximating'' bounded control)}
Alternative to Example~\ref{example:IOC3}, consider the non-quadratic cost function defined by
\begin{align}\label{eq:eta_bangbang}
\eta(r) &= \frac{e}{e^{1/r} - e}
\,,
\end{align}
which represents the control penalty, is infinitely flat near the origin and blows up at $r = 1$. This means that small control magnitudes are penalized very lightly, while magnitudes approaching 1 incur an infinite cost, effectively enforcing a hard input constraint.


The resulting optimal feedback law is given as
\begin{subequations}\label{eq:u*_bangbang}
\begin{align}
v^* = -\rho\varepsilon_1\frac{\sgn(L_{\bar{g}_1}V)}{1+\ln\left(1+ \dfrac{1}{\varepsilon_1 |L_{\bar{g}_1}V|}\right)}\\
\omega^* = -\varepsilon_2\frac{\sgn(L_{g_2}V)}{1+\ln\left(1+ \dfrac{1}{\varepsilon_2 |L_{g_2}V|}\right)}\,.
\end{align}
\end{subequations}
However, for similar reasons as in Example~\ref{example:IOC3}, to guarantee the upper bound $|v| \leq \overline{v}$ and $|\omega| \leq \overline{\omega}$, we choose
\begin{align}
\varepsilon_1(\rho) &= \frac{\overline{v}}{\sigma + \rho} > 0\,,
\end{align}
for small $\sigma > 0$ and $\varepsilon_2 = \overline{\omega}$. Fig~\ref{fig:ioc_controleffort_ex4} shows the simulation results with $\bar{v} = \bar{\omega} = 1$ and $\sigma = 0.1$. Notably, unlike the concentrated, high initial effort observed in Example~\ref{example:IOC3}, the control law \eqref{eq:u*_bangbang} produces a more evenly sustained control effort over time. This behavior is due to the cost function \eqref{eq:eta_bangbang}, which assigns minimal cost to small inputs, thereby tolerating low-magnitude control efforts over the trajectory. The settling times on the controls in Fig~\ref{fig:ioc_controleffort_ex4} are different due to this more sustained control effort. One certainly cannot expect the achievement of fixed-time (FxT) stabilization under bounded control, but only finite-time (FT) stabilization. 
\end{example}

\begin{figure}[t]
\centering
\begin{subfigure}{0.48\textwidth}
\centering
\includegraphics[width=\textwidth]{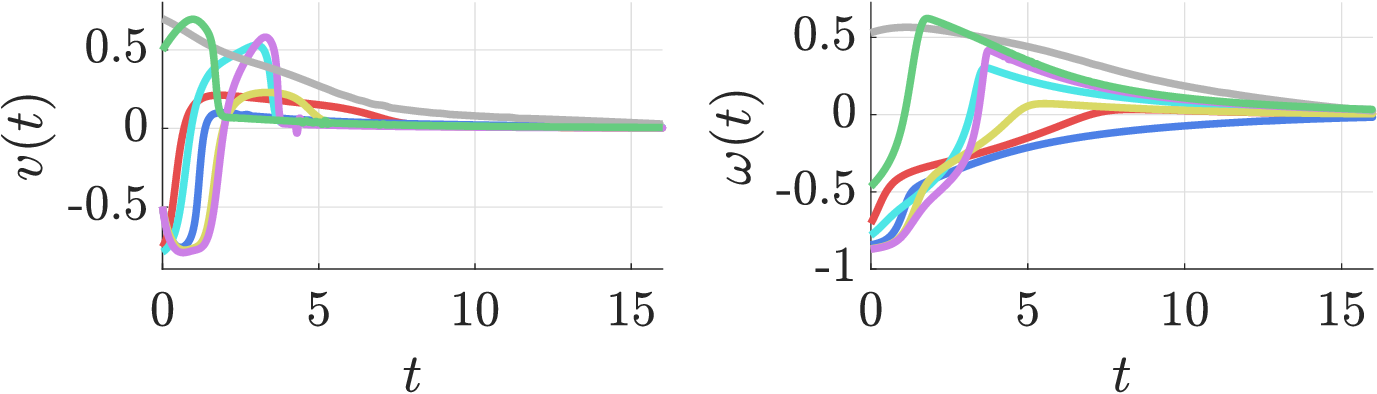}
\caption{Control effort of the bounded controller \eqref{eq:u*_ln}.}
\label{fig:ioc_controleffort_ex3}
\end{subfigure}
\vspace{0.25cm}
\begin{subfigure}{0.48\textwidth}
\centering
\includegraphics[width=\textwidth]{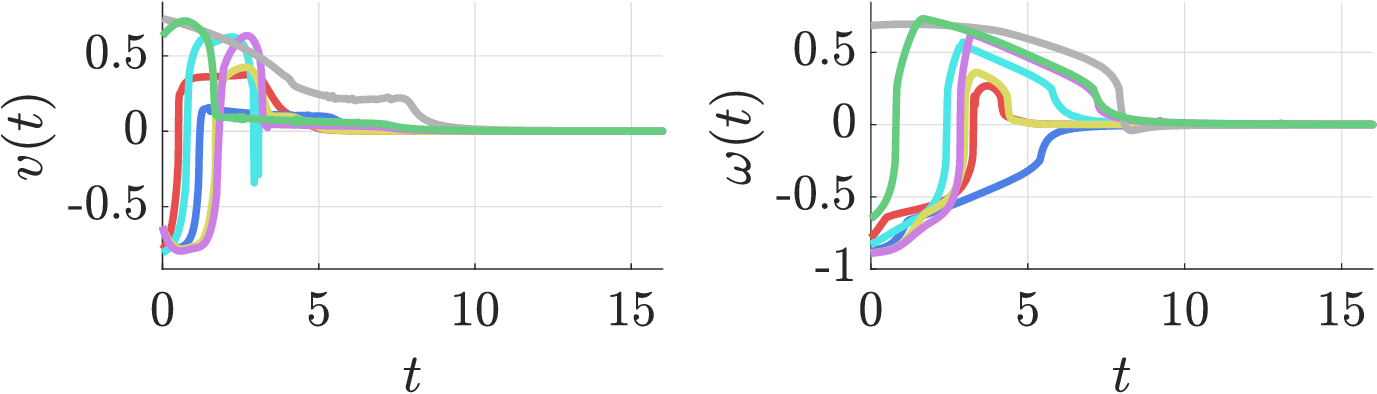}
\caption{Control effort of the bounded controller \eqref{eq:u*_bangbang}.}
\label{fig:ioc_controleffort_ex4}
\end{subfigure}
\captionsetup{belowskip=0pt}
\caption{Comparison in control effort between Example~\ref{example:IOC3} and Example~\ref{example:IOC4}  with $\bar v = \bar \omega = 1$.}
\label{fig:ioc_ex3_vs_ex4}
\end{figure}

\section{Gain margin and robustness to actuator saturation} 

The key practical benefit of the $L_gV$ inverse optimal control is the {\em gain margin}, i.e., the robustness to the uncertainty in the input coefficients. In the unicycle model, the coefficients on the inputs $(v,\omega)$ are taken, implicitly, as $(1,1)$. In the simplest of the inverse optimal controllers \eqref{eq-u*-quad0}, the gains $\varepsilon_1,\varepsilon_2>0$ are arbitrary, which means that the inverse optimal controller has a gain margin of $(0,\infty)$, greater than the conventional gain margin $[1/2,\infty)$ of optimal controllers for systems with drift. The unusually strong gain margin $(0,\infty)$, which allows the reduction of the optimal controller's gain for the unicycle to an {\em arbitrarily low} value is a consequence of the unicycle system being driftless. 

Inverse optimality doesn't afford only robustness to the uncertainty in the input coefficient. It also enables operation under arbitrarily low input saturation, as we have already observed in Examples~\ref{example:IOC3} and \ref{example:IOC4}. 
Despite the arbitrarily low input saturations, the feedback always acts in the direction determined by the sign of the $L_{\bar{g}_1}V$ and $L_{g_2}V$ terms (with full expressions given in~\eqref{eq-u*-quad0}). The inverse optimal controller always delivers the ``correct directions'' for control by design: they are made to do so by our construction of the CLF $V(\rho,\delta,\gamma)$. 

\section{Adaptive Stabilization under Unknown Input Coefficients}

Consider the model
\begin{subequations}
\label{eq:unicycle_polar_closed_loop-Gv-slip}
\begin{align}
\dot{\rho} &=  -b_1v \cos\gamma\\
\dot{\delta} &=  b_1\frac{v}{\rho}\sin(\gamma)
\\
\dot{\gamma} &= b_1\frac{v}{\rho}\sin(\gamma)
-b_2 \omega
\,,
\end{align}
\end{subequations}
where the constants $b_1,b_2$ are positive but otherwise completely unknown. For example, $b_1,b_2 \in (0,1]$ may physically represent {\em unknown} wheel slippage coefficients (on a two-wheeled mobile robot). 
The adaptive control laws
\begin{eqnarray}
\label{eq-v-adapt}
v &=& -\hat\varepsilon_1 \rho L_{\bar{g}_1}V\\
\label{eq-om-adapt}
\omega &=& -\hat\varepsilon_2 L_{g_2}V\,,
\end{eqnarray}
are introduced where $\hat \varepsilon_1, \hat \varepsilon_2$ are the online estimates of $1/b_1,1/b_2$, respectively. The derivative of the strict CLF \eqref{eq:globa_CLF} is given as
\begin{equation}
\dot V = -(L_{\bar g_1} V)^2 \left(1 - b_1 \tilde\varepsilon_1 \right) -(L_{ g_2} V)^2 \left(1 - b_2 \tilde\varepsilon_2 \right)\,,
\end{equation}
where $\tilde\varepsilon_i = 1/b_i - \hat\varepsilon_i$. 
Taking the adaptive CLF
\begin{equation}\label{eq:CLF_adapt}
V_{\rm a} = \ln(1+n(V)) + \frac{b_1}{2\mu_1}\tilde\varepsilon_1^2 + \frac{b_2}{2\mu_2}\tilde\varepsilon_2^2 \,,
\end{equation}
with adaptation gains $\mu_1,\mu_2>0$ and any normalization function $n\in\mathcal{K}_\infty\cap C^1$, including, for example, $n(V) = n_0V, n_0>0$, we get
\begin{align}
\dot V_{\rm a} =& -\frac{n'(V)}{1+n(V)}\left[(L_{\bar g_1} V)^2 +(L_{ g_2} V )^2 \right]
\nonumber\\
& +b_1\tilde\varepsilon_1\left[  \frac{n'(V)}{1+n(V)}(L_{\bar g_1} V)^2  - \frac{\dot{\hat\varepsilon}_1}{\mu_1}\right] \nonumber\\
& +b_2\tilde\varepsilon_2\left[  \frac{n'(V)}{1+n(V)}(L_{ g_2} V)^2  - \frac{\dot{\hat\varepsilon}_2}{\mu_2}\right]\,.
\end{align}
Hence, we pick the update laws
\begin{eqnarray}
\label{eq-eps1-adapt}
\dot{\hat\varepsilon}_1 & = & \mu_1 \frac{n'(V)}{1+n(V)}(L_{\bar g_1} V)^2\\
\label{eq-eps2-adapt}
\dot{\hat\varepsilon}_2 & = & \mu_2 \frac{n'(V)}{1+n(V)}(L_{ g_2} V)^2  \,,
\end{eqnarray}
and obtain
\begin{equation}\label{eq:adapt_dotV}
\dot V_{\rm a} =  -\frac{n'(V)}{1+n(V)}\left[(L_{\bar g_1} V)^2 +(L_{ g_2} V )^2 \right]  \,,
\end{equation}
which is negative for all $(\rho,\delta,\gamma)\neq (0,0,0)$ but is not negative definite in the overall state of the adaptive system, which also includes $(\tilde\varepsilon_1,\tilde\varepsilon_2)$. With LaSalle's theorem, we obtain the following result.

\begin{theorem} \label{thrm:adapt}
Consider the system \eqref{eq:unicycle_polar_closed_loop-Gv-slip} with arbitrary unknown $b_1,b_2>0$, along with the control law \eqref{eq-v-adapt}, \eqref{eq-om-adapt} and the update laws \eqref{eq-eps1-adapt}, \eqref{eq-eps2-adapt}. For the CLF \eqref{eq:globa_CLF}, with class $\mathcal{K}_\infty$ functions $\alpha_1,\alpha_2$ such that 
\begin{align}
\alpha_1\left(|\rho,\delta,\gamma|\right) \leq V(\rho,\delta,\gamma)
\leq \alpha_2\left(|\rho,\delta,\gamma|\right)\,,
\end{align}
for all initial conditions $(\rho(0),\delta(0),\gamma(0)) \in \{\rho > 0\}\times\mathbb{R}^2$ and for all parameter initial conditions $\hat \varepsilon_1(0)\in\mathbb{R},\hat \varepsilon_2(0)\in\mathbb{R}$, the following holds:
\begin{align}\label{eq:adapt_upsilon}
&\Upsilon(t) \leq a_1^{-1}\left( M\left(e^{m \, a_2(\Upsilon(0))} - 1\right)\right)\,,\quad \forall t \geq 0\,,
\end{align}
with $\Upsilon = 
|(\rho,\delta,\gamma,\tilde\varepsilon_1,\tilde\varepsilon_2)|$, where $a_1(r) = \min\{n\circ\alpha_1(r),r^2\}$ and $a_2(r) = \max\{n\circ\alpha_2(r),r^2\}$
are $\mathcal{K}_\infty$ functions and
\begin{align}
M =  
\max\left\{1,\dfrac{c_2}{c_1}\right\}
\,, & \quad
\label{eq:adapt_M}
m =  \max\left\{c_2,\frac{1}{c_2}\right\}
\,,
\\
\label{eq-adaptivec1c2}
c_1 =  \min\left\{\frac{b_1}{2\mu_1},\frac{b_2}{2\mu_2}\right\}
\,, &\quad
c_2 =  \max\left\{\frac{b_1}{2\mu_1},\frac{b_2}{2\mu_2}\right\}\,.
\end{align}
In addition, $\rho(t),\delta(t),\gamma(t)\rightarrow 0$ as $t\rightarrow\infty$. 
\end{theorem}

In Fig.~\ref{fig:adaptive_control_fig}, we compare the  quadratic cost $L_gV$ controller (red), given by $v = -\rho L_{\bar{g}_1}V$ and $\omega = -L_{g_2}V$, with the adaptive control law in \eqref{eq-eps1-adapt}--\eqref{eq-eps2-adapt}, which uses the normalization function $n(V) = V$, initial parameter estimates $\hat{\varepsilon}_1(0) = \hat{\varepsilon}_2(0) = 0$, and adaptation gains $\mu_1 = \mu_2 = 0.5$ (blue) and $\mu_1 = \mu_2 = 1$ (cyan). The system is initialized at $(\rho(0),\delta(0),\gamma(0)) = (1,-\pi/2,-\pi/2)$, under uncertainty in $b_1 = b_2 = 1$. For the quadratic cost $L_gV$ controller, this corresponds to assuming $b_1\varepsilon_1 = b_2\varepsilon_2 = 1$. The adaptive controller exhibits smaller peak values in both forward velocity and steering input (Fig.~\ref{fig:normalized_control_input}), even while achieving a faster decay of the state norm (Fig.~\ref{fig:adapt_state_norm}). This behavior is clearly due to $\hat{\varepsilon}_{1,2}$ starting small but quickly growing beyond $1$, thereby accelerating the decay rate to be faster than the $L_gV$ controller.

\begin{figure}[t]
\centering
\begin{subfigure}[b]{\linewidth}
\centering
\includegraphics[width=.85\linewidth]{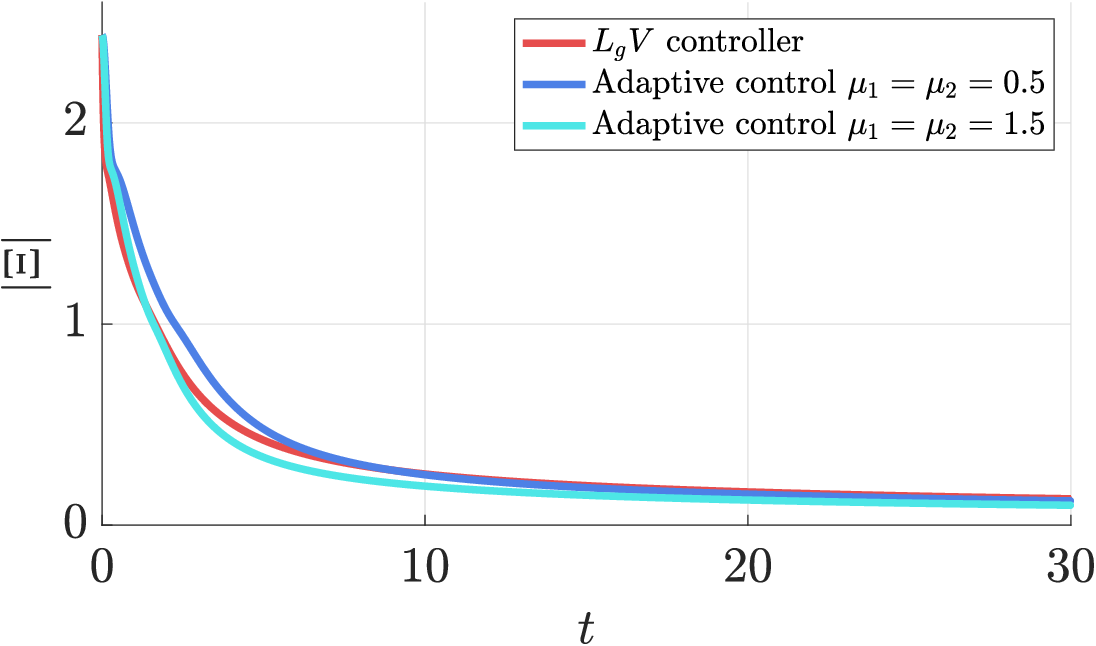}
\caption{Norm of $\Xi \coloneqq (\rho,\delta,\gamma)$ over time for each control law showing faster decay rate for both adaptive control laws.}
\label{fig:adapt_state_norm}
\end{subfigure}
\hfill
\vspace{0.1cm}
\begin{subfigure}[b]{\linewidth}
\centering
\includegraphics[width=.9\linewidth]{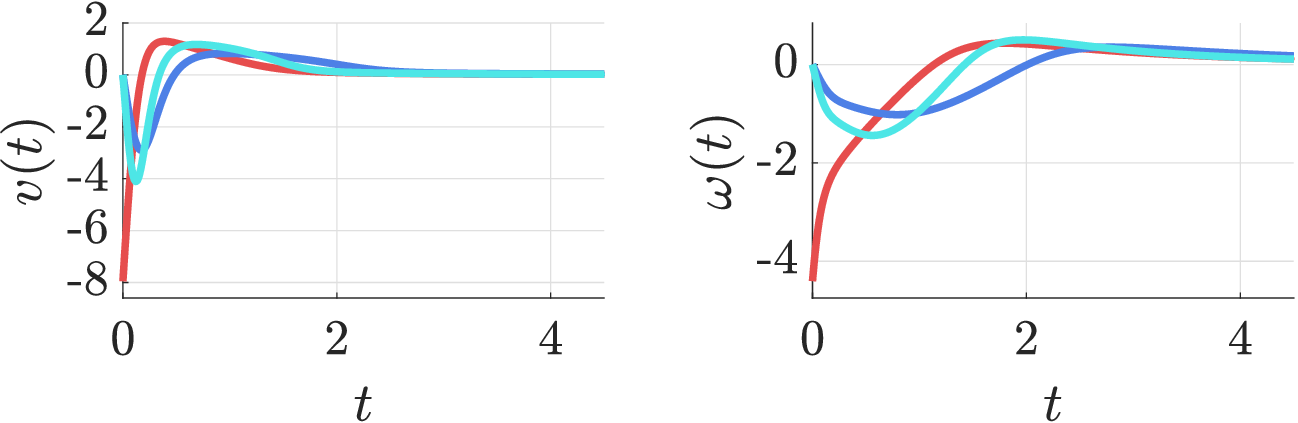}
\caption{Maximum magnitude of the control input is smaller for both adaptive control laws compared to the $L_gV$ controller despite the faster decay rate.}
\label{fig:normalized_control_input}
\end{subfigure}
\hfill
\vspace{0.1cm}
\begin{subfigure}[b]{\linewidth}
\centering
\includegraphics[width=.9\linewidth]{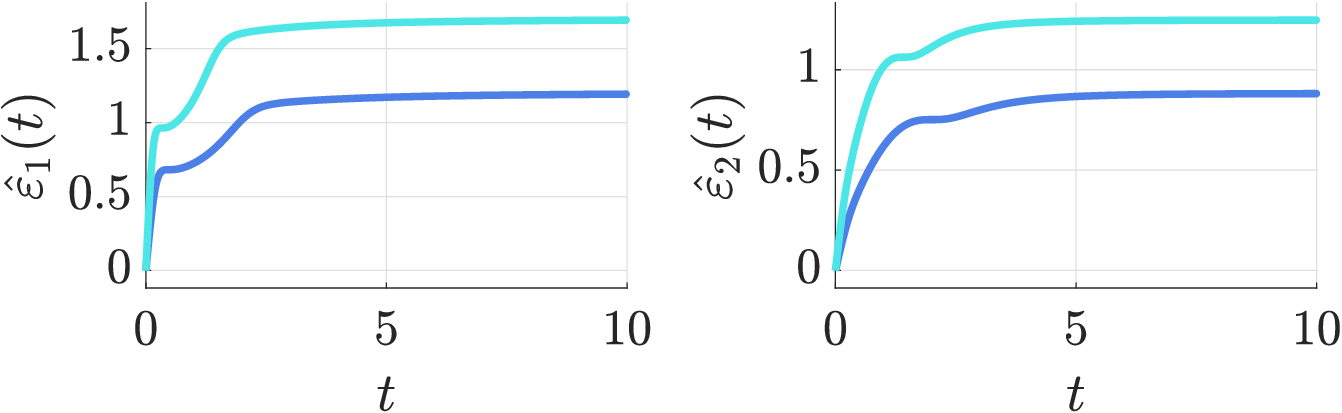}
\caption{$\hat{\varepsilon}_1$ and $\hat{\varepsilon}_2$ grows above $1$ leading to the faster decay rate.}
\label{fig:adaptive_eps}
\end{subfigure}
\caption{Comparison of state trajectories and control inputs under the quadratic cost $L_gV$ controller with both $b_1\varepsilon_1$ and $b_2\varepsilon_2$ taken as $1$ and the proposed adaptive controller with two different adaptation gains.}
\label{fig:adaptive_control_fig}
\end{figure}

The adaptive control law 
\eqref{eq-v-adapt}, \eqref{eq-om-adapt}, \eqref{eq-eps1-adapt}, \eqref{eq-eps2-adapt}
has similarities with the inverse optimal controller \eqref{eq-u*-quad0}. The difference is that the adaptive controller works even if it starts with gains of the wrong sign, $\hat\varepsilon_1(0)<0, \hat\varepsilon_2(0)<0$. Additionally, the adaptive controller's gain is improved, online, by learning from the transients reflected in $L_{\bar g}V$.

Of course, all update laws exhibit drift in the presence of disturbances. By adding leakage to the update laws \eqref{eq-eps1-adapt}, \eqref{eq-eps2-adapt}, as is standard, boundedness of the vehicle states and parameter estimates, as well as practical regulation of $\rho(t),\delta(t),\gamma(t)$ would be ensured. 


\section{Conclusion}

Leveraging globally strict CLFs for the unicycle parking problem, we establish an inverse optimal framework that yields controllers optimal with respect to meaningful costs without solving HJB equations. The proposed designs accommodate a wide range of penalties on control effort, including bounded inputs, while ensuring robustness through an infinite gain margin. Building on this property, we further develop an adaptive controller capable of addressing model uncertainty, with comparative results highlighting the effectiveness of both adaptive and non-adaptive implementations.

\bibliographystyle{IEEEtranS}
\bibliography{root}

\end{document}